\documentclass[11pt]{article}
\usepackage{graphicx}
\usepackage{amsmath,amsfonts,amsthm}
\usepackage{amssymb,latexsym}
\usepackage{fixmath}
\usepackage{mathrsfs,amsbsy}
\usepackage{dsfont}
\usepackage{enumerate}

\def\wtd{\widetilde}
\def\what{\widehat}

\DeclareMathOperator{\diag}{diag}

\DeclareMathOperator{\rank}{rank}

\DeclareMathOperator{\HH}{H}
\DeclareMathOperator{\T}{T}

\def\ba{\pmb{a}}
\def\bb{\pmb{b}}
\def\bc{\pmb{c}}

\def\be{\pmb{e}}

\def\bp{\pmb{p}}
\def\bq{\pmb{q}}

\def\bs{\pmb{s}}
\def\bt{\pmb{t}}

\def\bv{\pmb{v}}
\def\bw{\pmb{w}}
\def\bx{\pmb{x}}
\def\by{\pmb{y}}
\def\bz{\pmb{z}}

\newtheorem{proposition}{Proposition}[section]
\newtheorem{theorem}{Theorem}[section]
\newtheorem{lemma}{Lemma}[section]
\newtheorem{corollary}{Corollary}[section]

\theoremstyle{definition}

\newtheorem{remark}{Remark}[section]

\numberwithin{equation}{section}
\numberwithin{figure}{section}
\numberwithin{table}{section}

\allowdisplaybreaks

\usepackage{kbordermatrix}
\usepackage{caption}
\usepackage{amsfonts,amsthm}
\usepackage{amsmath,amssymb,enumerate,color}
\usepackage{algorithm,algorithmic}
\usepackage{epsfig,graphicx,psfrag,mathrsfs,subfigure}
\usepackage{eucal}
\usepackage{yhmath}
\usepackage{hyperref}
\hypersetup{colorlinks}
\usepackage{diagbox}
\RequirePackage{multirow}
\usepackage{textcomp}
\usepackage[]{fontenc}
\usepackage{extarrows}
\def\wtd{\widetilde}
\def\what{\widehat}
\usepackage{rotating}
\usepackage{lscape}
\usepackage{bm}
\usepackage{xspace}
\usepackage{tikz}
\usetikzlibrary{petri} 
\usetikzlibrary{shadows,arrows,positioning,shapes.geometric} 
\usetikzlibrary{calc} 
\usepackage{geometry}
 \usepackage{lipsum}

\def\ba{\pmb{a}}
\def\bb{\pmb{b}}
\def\bc{\pmb{c}}

\def\be{\pmb{e}}

\def\bp{\pmb{p}}
\def\bq{\pmb{q}}

\def\bs{\pmb{s}}
\def\bt{\pmb{t}}

\def\bv{\pmb{v}}
\def\bw{\pmb{w}}
\def\bx{\pmb{x}}
\def\by{\pmb{y}}
\def\bz{\pmb{z}}
\def\bzs{\mathbf{0}}

\def\diag{{\rm diag}}

\def\scrR{\mathscr{R}}

\def\wtd{\widetilde}
\def\what{\widehat}

\def\bbC{\mathbb{C}}

\def\bbP{\mathbb{P}}
\def\bbR{\mathbb{R}}
\def\bbS{\mathbb{S}}

\def\RE{\mathrm{Re}}
\def\IM{\mathrm{Im}}

\renewcommand{\algorithmicrequire}{\textbf{Input:}}
\renewcommand{\algorithmicensure}{\textbf{Output:}}

\numberwithin{equation}{section}
\numberwithin{figure}{section}
\numberwithin{table}{section}

\graphicspath{{figures/}{figure/}{pictures/}%
	{picture/}{pic/}{pics/}{image/}{images/}}

\title{Optimality Conditions for Rational Minimax  Approximations: Bridging Ruttan's Criteria to Dual-Based Methods
}
\author{Lei-Hong Zhang\thanks{School of Mathematical Sciences, Soochow University, Suzhou 215006, Jiangsu, China. This work was
 supported in part by the National Natural Science Foundation of China (NSFC-12471356, NSFC-12371380), Jiangsu Shuangchuang Project (JSSCTD202209) and  Academic Degree and Postgraduate Education Reform Project of Jiangsu Province.
        Email: {\tt longzlh@suda.edu.cn}.}
        }
 
 \date{\today}

\begin{document}

\maketitle

\begin{abstract}
This paper presents a theoretical discussion on Ruttan's optimality conditions for rational minimax approximations in discrete and continuum settings, integrating analytical foundations with computational practice. 
 We develop extended second-order optimality criteria for the discrete case, demonstrating that Ruttan's sufficient condition for global solutions [Ruttan, {\it Constr. Approx.}, 1 (1985), 287-296]  becomes necessary when the number of extreme points is minimal. Our analysis further uncovers fundamental relationships between these conditions and the dual-based {\tt d-Lawson} method [L.-H. Zhang et al., {\it Math. Comp.}, 94 (2025), 2457-2494], proving that strong duality in {\tt d-Lawson}  ensures simultaneous satisfaction of both Ruttan's and Kolmogorov's criteria. Additionally, we show that  minimax approximants on a continuum satisfying Ruttan's sufficient global optimality  can be captured through discrete minimax approximations at properly chosen boundary points, thereby enabling efficient computation of minimax approximants on a continuum using discrete methods.
\end{abstract}

\medskip
{\small
\medskip

} 
 
 
 \section{Introduction}\label{sec:intro}
The theory of minimax (a.k.a. Chebyshev, best or uniform) approximation originates in the pioneering work of Poncelet \cite{ponc:1835}, Chebyshev \cite{cheb:1854,cheb:1859} and Haar \cite{haar:1918}, who established the existence and uniqueness of the minimax real polynomial  $p\in \bbP_n$ with degree at most $n$  to a continuous function $f$ on an interval, where optimality is characterized by the well-known alternation theorem (see \cite[p. 75]{chen:1982}).
For compact subsets of  $\mathbb{C}$, the minimax complex polynomial   $p\in\mathbb{P}_n$ retains uniqueness and is characterized by Kolmogorov's necessary and sufficient condition \cite{kolm:1948a}. These results extend further to Haar-type subspaces in general normed spaces equipped with the supremum norm (see e.g., \cite{chen:1982,mein:1967,rice:1969,rish:1961,shap:1971,sing:1970}). 
For historical perspectives on approximation theory, readers may consult \cite{appr:web} 
as a valuable resource.

The optimality conditions for (local or global) rational minimax approximations are less complete than those for linear minimax approximations \cite{chen:1982,mein:1967,rice:1969,rish:1961,shap:1971,sing:1970}. These rational minimax approximations exhibit strong domain dependence, lacking universal necessary and sufficient conditions for characterizing global minimax solutions \cite{gutk:1983,sing:2006}. 
Furthermore, fundamental distinctions arise in the optimality criteria for (a) compact continuum domains versus discrete point sets, (b) real-valued versus complex-valued function approximations.
For the complex case, optimality conditions for the rational minimax approximations up to 1983 are surveyed in \cite{gutk:1983}.

Let the set of rational functions of type $(n_1,n_2)$ be\footnote{By convention, $q \equiv 1$ when $n_2 = 0$.}
$$\mathscr{R}_{(n_1,n_2)} \triangleq \left\{ \frac{p}{q} \, \bigg| \, p \in \mathbb{P}_{n_1},~ 0 \not\equiv q \in \mathbb{P}_{n_2} \right\}, ~n_1\ge 0, n_2\ge 0, $$
where $\mathbb{P}_n$ consists of complex polynomials of degree at most $n$. 
For a given irreducible $\xi(x)=p(x)/q(x)\in\scrR_{(n_1,n_2)}$, the {\it defect}  of $\xi(x)$ is defined by
\begin{equation}\label{eq:defect}
\upsilon(p,q)\triangleq\min(n_1-\deg(p),n_2-\deg(q)),
\end{equation} 
where  $\deg(p)$ and $\deg(q)$ denote the degree of  $p$ and $q$, respectively.  When $\upsilon(p,q)=0$, we say $\xi(x)=p(x)/q(x)$ is {\it non-degenerate}. 

Let $\mathcal{B} \subset \mathbb{C}$ be a compact set and $f\in {\bf C}({\cal B})$ be a continuous function on $\mathcal{B}$. Define the supremum norm $\|f\|_{\infty,\mathcal{B}} = \sup_{x \in \mathcal{B}} |f(x)|$, and consider the rational minimax approximation problem 
\begin{equation}\label{eq:bestf0}
\eta_{\cal B}\triangleq\inf_{\xi \in \mathscr{R}_{(n_1,n_2)}} \|f-\xi\|_{\infty,\mathcal{B}}. 
\end{equation}
For $\xi \in \mathscr{R}_{(n_1,n_2)}$, let $\zeta(\xi) \triangleq \sup_{x \in \mathcal{B}} |f(x) - \xi(x)|$ denote its maximum approximation error.   Walsh's foundational result \cite{wals:1931,wals:1969} ensures the existence of a minimax approximant $\hat \xi \in \mathscr{R}_{(n_1,n_2)}$ whenever $\mathcal{B}$ contains no isolated points, satisfying
$$\|f - \hat \xi \|_{\infty,\mathcal{B}} = \eta_{\cal B}.$$
However, this existence guarantee does not extend to discrete node sets\footnote{
The discrete topology is implicitly assumed whenever working with finite point sets in minimax approximation theory, distinguishing these problems fundamentally from the continuous case where the standard topology of $\mathbb{C}$ is used.} $\mathcal{X} = \{x_j\}_{j=1}^m$ ($m\ge n_1+n_2+2$). Indeed, the infimum $\eta_{\cal B}$ of \eqref{eq:bestf0} may be unattainable, and even when attainable, the minimax approximation need not be unique \cite{natr:2020,tref:2019a}. Furthermore, uniqueness remains non-guaranteed even for a Jordan domain ${\cal B}$, and local best approximants may exist \cite{gutk:1983,sava:1977,this:1993,wals:1969,will:1979}.

As mentioned, optimality conditions for (local or global) rational minimax approximations depend critically on the domain structure. To present our results more clearly, we distinguish the domain of the approximated function $f$ using different notations: 
\begin{itemize}
\item  $\mathcal{B} \subset \mathbb{C}$ represents a general compact set; it can be a continuum, or a set containing finite points. ${\bf C}({\cal B})$ is the set of continuous functions on ${\cal B}$.

\item $\Omega \subset \mathbb{C}$ denotes a compact continuum with boundary $\Gamma := \partial \Omega$. Typical examples include closed intervals in real or imaginary axis, the unit circle, or simply connected domains bounded by Jordan curves. We denote by ${\bf C}_{A}(\Omega)$ the class of functions that are continuous on $\Omega$ and analytic in its interior whenever $\mathrm{int}(\Omega) \neq \emptyset$.

\item $\mathcal{X} = \{x_j\}_{j=1}^m$ ($m\ge n_1+n_2+2$) is a  discrete set. 
In dealing with approximation on $f\in {\bf C}_{A}(\Omega)$, by maximum modulus principle, efficient numerical methods are usually employed to solve \eqref{eq:bestf0} with carefully chosen discrete nodes ${\cal B}={\cal X}\subset \Gamma$.  
\end{itemize}

For a general compact ${\cal B}$, if the infimum $\eta_{\cal B}$ of \eqref{eq:bestf0} is attainable by $\hat \xi$, we
denote the error function by 
$\hat e(x)  \triangleq f(x)-\hat\xi(x)$ and the maximum error by $\|\hat e \|_{\infty,{\cal B}}$. 
The set of extreme points (a.k.a. reference points) are then
 \begin{equation}\label{eq:extremalsetY}
 {\cal E}(\hat \xi)\triangleq\left\{x\in {\cal B}:\left|f(x)- \hat \xi(x)\right|=\|\hat e \|_{\infty,{\cal B}}\right\}.
\end{equation}
The set ${\cal B}$ can be interpolated as $\Omega$, the boundary  $\Gamma$ or the discrete set ${\cal X}$ usually sampled in  $\Gamma$.

{\it Motivation and contributions.}
While optimality conditions for rational minimax approximations have been studied in the literature, the majority of existing results focus on theoretical analyses under the   assumption that ${\cal B}$ constitutes an infinite or continuum set \cite{gutk:1983,rutt:1985,this:1993}.
Recent years, particularly since 2018, have seen significant advances in efficient numerical methods for discrete rational approximations, offering new computational approaches. Notable developments include the adaptive Antoulas-Anderson (AAA) algorithm \cite{nase:2018,hutr:2023,drnt:2024} and its variant AAA-Lawson \cite{fint:2018}, surveyed extensively in \cite{nast:2023,natr:2026}.  Other methods, such as the rational Krylov fitting (RKFIT) \cite{begu:2017,gogu:2021}, the Loeb algorithm \cite{loeb:1957}, the SK iteration \cite{sako:1963}, and its stabilized version \cite{hoka:2020}, have also proven highly effective for computing rational (not necessarily minimax) approximations, with demonstrated success across diverse applications \cite{coop:2007,limp:2022,widt:2022,zhzy:2025}. 
Meanwhile,  building on a new dual-based framework, recent advancements in Lawson-type iterations \cite{yazz:2023,zhha:2025,zhyy:2025,zhzz:2025,zhzh:2025} have further expanded the algorithmic landscape for discrete rational minimax approximation \eqref{eq:bestf0}.

As previously noted, discrete rational minimax approximations \eqref{eq:bestf0} are not guaranteed to admit a solution, and their optimality conditions often depend on the domain. An example arises from Ruttan's necessary conditions \cite[Theorems 1.3 and 1.4]{rutt:1985}, derived from the second-order variational analysis (specifically, the $\lambda^2$-term $\varsigma$ in \eqref{eq:variation}). Although originally formulated for arbitrary compact sets $\mathcal{B}$, these conditions were later shown to fail for general continuum subsets of $\bbC$ \cite{this:1993}.

Based on these considerations, we conduct a systematic review of the local optimality conditions for rational minimax approximations \eqref{eq:bestf0}, synthesizing established relevant results from the literature while incorporating essential corrections. Our analysis particularly focuses on how the local necessary conditions determine their applicability across distinct classes of $\mathcal{B}$.  For a  local irreducible  solution $\hat{\xi}(x) = \hat{p}(x)/\hat{q}(x)$ of \eqref{eq:bestf0}, the argument hinges on the variational expression \cite{kolm:1948a,rutt:1985} for sufficient small $\lambda\in \bbR$:
\begin{equation}
\label{eq:variation}
\delta_{\lambda}(x) \triangleq |f(x) - \xi_{\lambda}(x)|^2 - |f(x) - \hat{\xi}(x)|^2 = \frac{\lambda \kappa(x) + \lambda^2   \varsigma(x) + O(\lambda^3)}{|q_\lambda(x)|^2}, \quad x \in \mathcal{B},
\end{equation}
where $\xi_{\lambda}(x) = p_{\lambda}(x)/q_{\lambda}(x)$ is a parameterized perturbation of $\hat{\xi}(x)$, defined via
$$p_{\lambda}(x) = \hat{p}(x) + \lambda s(x) + \lambda^2 a(x) \in \bbP_{n_1}, \quad
q_{\lambda}(x) = \hat{q}(x) + \lambda t(x) + \lambda^2 b(x) \in \bbP_{n_2}.  $$
Here $s, a \in \mathbb{P}_{n_1}$ and $t, b \in \mathbb{P}_{n_2}$.
The first-order analysis (governed by the $\lambda$-term $\kappa(x)$) leads to the Kolmogorov criterion, while second-order analysis (via the $\lambda^2$-term $\varsigma(x)$) furnishes Ruttan's optimality conditions. To derive their dual forms,  the classical Carathéodory Lemma \cite[Lemma 2.3.1]{shap:1971} and linear programming duality \cite[Chapters 12 and 13]{nowr:2006} are used. A summary of these local optimality conditions is given in Section \ref{subsec:summary}. This constitutes one of contributions of this paper. 

Compared to local optimality conditions, sufficient conditions for global rational minimax approximants remain scarce and often impractical for numerical implementation.
Within the new dual-based framework, a Lawson-type iteration, the {\tt d-Lawson} algorithm, was proposed in \cite{zhyy:2025} and the convergence is recently developed in \cite{zhha:2025}. {\tt d-Lawson} is a method for solving the max-min type dual problem of the original discrete minimax problem \eqref{eq:bestf0}.  Crucially, Ruttan's sufficient condition \cite[Theorem 2.1]{rutt:1985} for global optimality emerges as the theoretical linchpin for {\tt d-Lawson}'s efficacy. In the second part of this paper, we   establish 
\begin{itemize}
\item[(1)] the relationship between Ruttan's sufficient global optimality and a key property in the max-min dual problem: strong duality; 
\item[(2)] a result  that Ruttan's sufficient global optimality becomes necessary for $\hat \xi$ when the number of extreme points is minimal (i.e., $|\mathcal{E}(\hat{\xi})| = n_1 + n_2 + 2 - \upsilon(\hat{p},\hat{q})$); 
\item[(3)] that strong duality ensures simultaneous fulfillment of both Ruttan's and Kolmogorov's criteria.
\end{itemize}
These results constitute our second contribution.

Our final theoretical contribution determines a sufficient condition under which the rational minimax approximant $\hat{\xi}_{\Omega}$ of $f\in {\bf C}_{A}(\Omega)\setminus \scrR_{(n_1,n_2)}$ on a Jordan domain $\Omega \subset \mathbb{C}$ can be obtained through discrete approximation using suitable nodes $\mathcal{X}$ sampled in its Jordan curve $\Gamma=\partial \Omega$. This highlights a fundamental difference between rational and polynomial approximation: while the polynomial minimax solution $\hat{\xi}_{\Omega}$ of \eqref{eq:bestf0} with $\mathcal{B}=\Omega$ can always be exactly reconstructed (or approximated arbitrarily well) from its discrete version \eqref{eq:bestf0} with $\mathcal{B}=\mathcal{X}$ when $\mathcal{X}$ contains the extreme points in an extremal signature \cite{rish:1961} (or meets appropriate density conditions on $\Gamma = \partial\Omega$ \cite[Chapter 3]{chen:1982}), the rational case exhibits fundamentally different behavior - counterexamples exist where $\hat{\xi}_{\Omega}$ cannot be recovered (nor even closely approximated) by any discrete solution $\hat{\xi}_{\mathcal{X}}$ with $\mathcal{X} \subset \Gamma$, regardless of sampling density (see Section \ref{sec:boundcont}). We establish a   sufficient condition ensuring that $\hat{\xi}_{\Omega}$ indeed solves an appropriate discrete minimax approximation problem with ${\cal X}\subseteq \Gamma$, thereby providing a theoretical basis for practical computation of continuum minimax approximants via discrete methods.

{\it Paper organization.} This paper is organized as follows. Section \ref{sec:localOpt} examines optimality conditions for local minimax approximants, reviewing both polynomial and rational cases, including Kolmogorov's criteria (first-order conditions) and Ruttan's criteria (second-order conditions). In Section \ref{sec:RuttanSuf}, we focus on Ruttan's sufficient condition for global solutions of \eqref{eq:bestf0} in discrete minimax problems, demonstrating its necessity when $\hat{\xi}$ has the minimal number of extreme points ($|\mathcal{E}(\hat{\xi})| = n_1 + n_2 + 2 - \upsilon(\hat{p},\hat{q})$). Section \ref{sec:dlawsonFound} concerns the theoretical foundation for the {\tt d-Lawson} iteration by showing that strong duality ensures both Ruttan's and Kolmogorov's criteria. Subsequently, Section \ref{sec:boundcont} explores the connection between rational minimax approximations on continua and discrete sets, presenting a sufficient condition to recover continuum minimax approximants from discrete ones. Finally, concluding remarks are summarized in Section \ref{sec:conclude}.

{\it Notation}.
We adhere to the notation used in \cite{zhyy:2025}. The imaginary unit is denoted by $\mathtt{i} = \sqrt{-1}$. For a complex number $\mu = \mu^{\mathtt{r}} + \mathtt{i} \mu^{\mathtt{i}} \in \bbC$, we define its modulus and complex conjugate as 
$
|\mu| = \sqrt{(\mu^{\mathtt{r}})^2 + (\mu^{\mathtt{i}})^2}, \quad \overline{\mu} = \mu^{\mathtt{r}} - \mathtt{i} \mu^{\mathtt{i}},
$
where $\RE(\mu) = \mu^{\mathtt{r}} \in \bbR$ and $\IM(\mu) = \mu^{\mathtt{i}} \in \bbR$ represent the real and imaginary parts of $\mu$, respectively. 
Column vectors are denoted by bold lowercase letters. The symbol $\bbC^{n \times m}$ (resp. $\bbR^{n \times m}$) denotes the set of all $n \times m$ complex (resp. real) matrices, where the identity matrix is given by $I_n \equiv [\be_1, \dots, \be_n] \in \bbR^{n \times n}$, with $\be_i$ being its $i$-th column. 
For a vector $\bx \in \bbC^n$, we define $\diag(\bx) = \diag(x_1, \dots, x_n)$ as the diagonal matrix formed from $\bx$, and $\|\bx\|_2$ as the vector $2$-norm.
Given $\bx, \by \in \bbC^n$ with $y_j \neq 0$ ($1 \leq j \leq n$), their element-wise division is  $
\bx ./ \by = \left[ x_1/y_1, \dots, x_n/y_n \right]^{\T}.
$ 
For a matrix $A \in \bbC^{m \times n}$,  $\operatorname{span}(A)$ denotes its column space, and $A^{\HH}$ (resp. $A^{\T}$) represents its conjugate transpose (resp. transpose).
Following MATLAB conventions,  $[A]_{i,j}$ refers to the entry of   $A$ located at the $i$-th row and $j$-th column.
Finally, we set
\begin{equation}\label{eq:basesP12}
\bbP_{n_1}={\rm span}(\psi_0(x),\dots,\psi_{n_1}(x)) ~{\rm and} ~\bbP_{n_2}={\rm span}(\phi_0(x),\dots,\phi_{n_2}(x)).
\end{equation}
A canonical choice is $\psi_j(x)=\phi_j(x)=x^{j-1}$.

 \section{Necessary optimality for local approximants}\label{sec:localOpt}

 \subsection{The Kolmogorov conditions: results from first-order analysis}

Building on the variational formulation in \eqref{eq:variation}, Kolmogorov-type optimality criteria arise from analyzing the first-order term with respect to $\lambda$. We first consider the polynomial minimax approximation case.
 \subsubsection{The polynomial case}\label{subsec:KolmP}
For polynomial minimax approximations (i.e., when $n_2 = 0$), Kolmogorov \cite{kolm:1948a} established a necessary and sufficient condition characterizing the unique minimax approximant $\hat{p} \in \mathbb{P}_{n_1}$. The analysis arises from considering perturbed polynomials of the form
$$p_{\lambda}(x) = \hat{p}(x) + \lambda s(x) \quad \forall s \in \mathbb{P}_{n_1}, \; \forall \lambda>0,$$
where $s$ is an arbitrary polynomial perturbation. The first-order variation of the squared error is then given by
$$\delta_{\lambda}(x) = |f(x) - p_{\lambda}(x)|^2 - |f(x) - \hat{p}(x)|^2 = -2\lambda \, \RE\left(\overline{(f(x) - \hat{p}(x))}s(x)\right) +  {O}(\lambda^2).$$
From this,   Kolmogorov \cite{kolm:1948a} obtained the primal form\footnote{Analysis with  $\lambda\in \bbR$ leads further to $$\max_{x\in {\cal E}(\hat{p})}\RE\left(\overline{(f(x) - \hat{p}(x))}s(x)\right)\ge 0\ge \min_{x\in {\cal E}(\hat{p})}\RE\left(\overline{(f(x) - \hat{p}(x))}s(x)\right), ~~\forall s\in \bbP_{n_1}.$$}    
\begin{equation}\label{eq:kolvP}
\max_{x\in {\cal E}(\hat{p})}\RE\left(\overline{(f(x) - \hat{p}(x))}s(x)\right)\ge 0, ~~\forall s\in \bbP_{n_1},
\end{equation}
where $ {\cal E}(\hat{p})$ is the set of extreme points defined in \eqref{eq:extremalsetY}. 

Condition \eqref{eq:kolvP} implies that  $\mathbf{0}\in {\rm conv}({\cal A})$ (cf. \cite[Theorem 2.3.2]{shap:1971} and  \cite[Lemma 1]{chlo:1964}), where $${\cal A}= \left\{\overline{(f(x) - \hat{p}(x))}\left[\begin{array}{c}\psi_0(x) \\\vdots \\\psi_{n_1}(x)\end{array}\right]\in \bbC^{n_1+1}:  x\in {\cal E}(\hat{p})\right\}.$$
The dual form of   Kolmogorov's condition relies on $\mathbf{0}\in {\rm conv}({\cal A})$ and the Carath\'eodory lemma:
\begin{lemma}[Carath\'eodory]\label{lem:cara}
If ${\cal F}\subset \bbR^n$,  then every point of the convex hull 
$
{\rm conv}({\cal F})=\left\{\sum_{j} \omega_j\by_j :\sum_{j} \omega_j=1,~\omega_j\ge 0,~ \by_j\in {\cal F}\right\}
$ can be written as a convex linear combination of at most $n +1$  points of ${\cal F}$. 
\end{lemma}
With the isomorphism $\bbC^{n_1+1}\cong \bbR^{2(n_1+1)}$, we have the Kolmogorov condition \cite{kolm:1948a} in the dual form (cf. \cite[Theorem 2.3.2]{shap:1971}): {\it $\hat p\in \bbP_{n_1}$ is the polynomial minimax  approximant of $ f\in {\bf C}({\cal B})$ if and only if there exist  $r$ ($1\le r\le 2n_1+3$) extreme points $z_1,\dots,z_r \in {\cal E}(\hat{p})$ and positive number $\omega_1,\dots,\omega_r$ with $\sum_{j=1}^r\omega_j=1$ such that\footnote{Condition \eqref{eq:kolvPD} implies that there is a non-null positive measure whose support consists of $r$ points of ${\cal E}(\hat p)$ such that $f(x) - \hat{p}(x)$ ``annihilates" $\bbP_{n_1}$ (\cite[Corollary 1]{rish:1961} and \cite[Chapter 2]{shap:1971}). Due to this fact,  condition \eqref{eq:kolvPD} (or   \eqref{eq:kolvP}) is also referred to as that $f(x) - \hat{p}(x)$ is orthogonal  to $\bbP_{n_1}$.} 
\begin{equation}\label{eq:kolvPD}
\sum_{j=1}^r \omega_j\overline{(f(z_j) - \hat{p}(z_j))}s(z_j)=0,~~\forall s\in \bbP_{n_1}.
\end{equation}
}

\subsubsection{The rational case}\label{subsec:KolmR}
Kolmogorov's conditions have been extended to  the nonlinear minimax approximations (see e.g., \cite{brae:1973,gutk:1983} and \cite[Theorem 86]{mein:1967}). For the rational minimax approximation, suppose $\hat \xi= {\hat p}/{\hat q}\in   \scrR_{(n_1,n_2)}$ is a (local) minimax irreducible rational approximant to $f\in {\bf C}({\cal B})$  with {defect} $\upsilon(\hat p,\hat q)$. Let   $n\triangleq n_1+n_2-\upsilon(\hat p,\hat q).$ In \cite[Theorem 3]{will:1979} (see also \cite[Theorem 1.2]{rutt:1985}), it analyzes the difference  
\begin{equation}\label{eq:varR1}
\delta_{\lambda}(x) =|f(x) -  {\xi}_{\lambda}(x)|^2 - |\hat e(x)|^2=[\lambda \kappa(x)+\lambda^2 \rho(x)]/|q_\lambda(x)|^2, 
\end{equation}
where $\hat e(x)=f(x)-\hat \xi(x)$, $  \xi_{\lambda}(x)=  \frac{p_{\lambda}(x)}{q_{\lambda}(x)}$, $p_{\lambda}(x)=\hat p(x)+
\lambda s(x),~  q_{\lambda}(x)=\hat q(x)+\lambda t(x)$, 
\begin{align} \label{eq:varRkrho}
 \kappa(x)&\triangleq {2 \RE\left(\overline{\hat e(x)\hat q(x)}g(x)/\hat q(x)\right)}, ~\rho(x)\triangleq { |f(x)t(x)-s(x)|^2-|f(x)t(x)-\hat \xi(x)t(x)|^2},\\\label{eq:varRg}
 g(x)&\triangleq t(x)\hat p(x)-s(x)\hat q(x)\in \bbP_{n}.
\end{align}
Note that any $g\in \bbP_{n}$ can be expressed in the form of \eqref{eq:varRg} with certain pair $(s,t)\in \bbP_{n_1}\times \bbP_{n_2}$.  
Since the first-order term $\kappa(x)$ dominates $\delta_{\lambda}(x)$ for sufficiently small $\lambda>0$, in a similar way to deriving \eqref{eq:kolvP} and its dual form \eqref{eq:kolvPD},  the following theorem  can be established, describing the primal \cite[Theorem 3]{will:1979} and dual optimality conditions (see  \cite{rutt:1985} and also  \cite[Theorems 1 and 2]{this:1993}).
 \begin{theorem}\label{thm:kolR}
 Suppose $\mathcal{B} \subset \mathbb{C}$ is a   compact set containing at least $N=n_1+n_2+2$ distinct points. Let $\hat \xi =\frac{\hat p}{\hat q}\in   \scrR_{(n_1,n_2)}$ be an  irreducible rational (local) minimax approximant to $f\in {\bf C}({\cal B})\setminus \scrR_{(n_1,n_2)}$  of \eqref{eq:bestf0} with {defect} $\upsilon(\hat p,\hat q)$ given in \eqref{eq:defect}. Let $n=n_1+n_2-\upsilon(\hat p,\hat q).$ Then\footnote{The proof of \cite[Theorem 1.2]{rutt:1985} remains valid for ${\cal B}$ being a finite discrete set   equipped with the discrete topology.}
 \begin{itemize}
 \item[i)] (primal form) 
 \begin{equation}\label{eq:kolvR}
\max_{x\in {\cal E}(\hat \xi)}\RE\left(\overline{(f(x) - {\hat \xi}(x))\hat q(x)}g(x)/\hat q(x)\right)\ge 0, ~~\forall g\in \bbP_{n};
\end{equation}
\item [ii)] (dual form) there exist  $r$ ($n+2\le r\le 2n+3$) extreme points $z_1,\dots,z_r \in {\cal E}(\hat \xi)$ and positive number $\omega_1,\dots,\omega_r$ with $\sum_{j=1}^r\omega_j=1$ such that 
\begin{equation}\label{eq:kolvRD}
\sum_{j=1}^r \omega_j\overline{(f(z_j) - {\hat \xi}(z_j))\hat q(x_j)}g(z_j)/\hat q(z_j)=0,~~\forall g\in \bbP_{n}.
\end{equation}
 \end{itemize}
 \end{theorem}

 \begin{remark} We have the following statements for Theorem \ref{thm:kolR}:
 \item[1)] Theorem \ref{thm:kolR} reduces to the Kolmogorov conditions \eqref{eq:kolvP} and \eqref{eq:kolvPD} for the polynomial case (i.e., $n_2=0$ and $q\equiv 1$).
 \item [2)]
 A corollary from the dual form of \eqref{eq:kolvRD} is that for any (local)  irreducible rational minimax approximant $\hat \xi$ to $f\in {\bf C}({\cal B})\setminus \scrR_{(n_1,n_2)}$, ${\cal E}(\hat \xi)$ contains at least $d\ge n+2$ extreme points. 
 \item[3)] When $\mathcal{B}$ is a real interval, the rational minimax approximant $\hat{\xi}$ can be characterized by a necessary and sufficient condition—the classical alternation theorem \cite{chlo:1964}. However, no such characterization is known for general compact sets $\mathcal{B} \subset \mathbb{C}$ (see \cite[Section 2]{gutk:1983}). Indeed, \eqref{eq:kolvR} is not a sufficient condition for the local rational minimax  approximant (cf. \cite[Example 1]{will:1979}).

  \item [4)] The dual formulations \eqref{eq:kolvPD} and \eqref{eq:kolvRD} guarantee the existence of $r$ positive weights $\omega_j$ and their corresponding extreme points $x_j$, but they are not conducive to practical computation. Fortunately, for the discrete case ($\mathcal{B} = \mathcal{X}$), recent dual-based numerical methods—proposed in \cite{yazz:2023,zhha:2025,zhyy:2025}—offer efficient algorithms to determine these values and thereby solve the original minimax approximation problem \eqref{eq:bestf0} through its max-min dual problem. Additional details will be discussed in Section~\ref{sec:dlawsonFound}.
 \end{remark}  
 
\subsection{Ruttan's necessary conditions: results from second-order analysis}\label{subsec:Ruttan2nd}
To continue the variational analysis on $\delta_\lambda(x)$ \eqref{eq:varR1}, Ruttan in \cite{rutt:1985} considered the term associated with $\lambda^2$. In particular, a new perturbation ${\xi}_{\lambda}(x) = \frac{p_{\lambda}(x)}{q_{\lambda}(x)}$ is proposed where 
\begin{equation}\label{eq:pqRuttan}
p_{\lambda}(x) = \hat{p}(x) + \lambda s(x) + \lambda^2 a(x)\in \bbP_{n_1}, \quad q_{\lambda}(x) = \hat{q}(x) + \lambda t(x) + \lambda^2 b(x)\in \bbP_{n_2},~\lambda>0,
\end{equation}
$s, a \in \mathbb{P}_{n_1}$ and $t, b \in \mathbb{P}_{n_2}$. The introduction of the extra polynomials $a$ and $b$ is motivated by the requirement that the $\lambda$-dependent term should be annihilated on $\mathcal{E}(\hat{\xi})$ by the pair $(s,t)$. Consequently, $a$ and $b$ act as free parameterized polynomials. With \eqref{eq:pqRuttan}, it follows that, for sufficiently small $\lambda\in \bbR$,
\begin{equation}\label{eq:varR2}
\delta_{\lambda}(x) =|f(x) -  {\xi}_{\lambda}(x)|^2 - |\hat e(x)|^2=[\lambda \kappa(x)+\lambda^2 (\rho(x)+ \alpha(x))+O(\lambda^3)]/|q_\lambda(x)|^2, 
\end{equation}
where $\kappa(x),\rho(x)$ are given by \eqref{eq:varRkrho} and 
\begin{equation}\label{eq:alpha}
\alpha(x)\triangleq {2 \RE\left(\overline{\hat e(x)\hat q(x)}c(x)/\hat q(x)\right)},~c(x)\triangleq b(x)\hat p(x)-a(x)\hat q(x)\in \bbP_n.
\end{equation}
To ensure the $\lambda^2$-term $\rho(x) + \alpha(x)$ (corresponding to $\varsigma(x)$ in \eqref{eq:variation}) dominates $\delta_\lambda(x)$ for sufficiently small $\lambda > 0$, Ruttan~\cite{rutt:1985} imposed constraints on the polynomial pair $(s,t)$ to eliminate the $\lambda$-term $\kappa(x)$ uniformly for all $x \in \mathcal{E}(\hat{\xi})$. Specifically, we consider perturbation pairs $(s,t) \in \mathbb{P}_{n_1} \times \mathbb{P}_{n_2}$ in
\begin{equation}\label{eq:stconstrant}
{\bbS}\triangleq\left\{(s,t)\in \bbP_{n_1}\times \bbP_{n_2}:  \kappa(x)=0,~\forall x\in {\cal E}(\hat \xi)\right\}.
\end{equation}
Based on this observation, Ruttan established both \cite[Theorem 1.3]{rutt:1985} and its dual formulation \cite[Theorem 1.4]{rutt:1985}. These results were then used to derive a   necessary and sufficient condition \cite[Theorem 2.2]{rutt:1985} characterizing the global minimax approximant for the case where $|\mathcal{E}(\hat{\xi})|=n+2$. However, \cite{this:1993} later demonstrated that \cite[Theorem 1.3]{rutt:1985} fails for arbitrary compact subsets $\mathcal{B} \subset \mathbb{C}$, even for the canonical case of the unit disc $\mathcal{D} \triangleq \{x \in \mathbb{C} : |x| \leq 1\}$ (cf. \cite[Example 5]{this:1993}).

Interestingly, our next analysis shows that Ruttan's theoretical framework in \cite{rutt:1985} remains valid for arbitrary finite discrete sets $\mathcal{B} \subset \mathbb{C}$. This validity provides crucial theoretical support for numerical approximation methods, particularly the recently developed {\tt d-Lawson} iteration \cite{zhha:2025,zhyy:2025} that operates on discrete node sets $\mathcal{X}=\{x_j\}_{j=1}^m$.

\begin{theorem}[Ruttan's primal formulation of second-order optimality conditions]\label{thm:RuttanP}
Suppose $\mathcal{B}={\cal X}=\{x_j\}_{j=1}^m$ containing $m\ge N=n_1+n_2+2$ distinct points.   Let $\hat \xi =\frac{\hat p}{\hat q}\in   \scrR_{(n_1,n_2)}$ be an irreducible  rational (local) minimax  approximant to $f\in {\bf C}({\cal B})\setminus \scrR_{(n_1,n_2)}$  of \eqref{eq:bestf0} with {defect} $\upsilon(\hat p,\hat q)$. Let $n=n_1+n_2-\upsilon(\hat p,\hat q).$ Then we have \eqref{eq:kolvR} and \eqref{eq:kolvRD}. Moroever,  for any $(s,t)\in {\bbS}$, it holds that 
\begin{equation}\label{eq:RuttanR2}
\left\{\min_{(a,b)\in \bbP_{n_1}\times \bbP_{n_2}}\max_{x\in {\cal E}(\hat \xi)}  (\rho(x) + \alpha(x))\right\}\ge 0,
\end{equation}
where ${\bbS}$ and  $\rho(x) + \alpha(x)$ are given in  \eqref{eq:stconstrant} and \eqref{eq:varR2}, respectively.
\end{theorem}
\begin{proof}
Following Ruttan's framework~\cite{rutt:1985} with necessary adaptations, we proceed by contradiction. Assume that there exist $(s,t)\in\bbS$ and $(a,b)\in\mathbb{P}_{n_1}\times\mathbb{P}_{n_2}$ satisfying
$$\left\{\max_{x\in\mathcal{E}(\hat{\xi})} \big(\rho(x) + \alpha(x)\big)\right\} < 0.$$
This  implies the existence of $\epsilon_1>0$ such that
$$\rho(x) + \alpha(x) < -\epsilon_1 \quad \forall x\in\mathcal{E}(\hat{\xi}).$$
For the parameterized perturbation $\xi_\lambda = p_\lambda/q_\lambda \in \mathscr{R}_{(n_1,n_2)}$ defined in~\eqref{eq:pqRuttan} and sufficiently small $\lambda\in\mathbb{R}$, we observe that $q_\lambda(x) \neq 0$ for all $x\in\mathcal{X}$, and  the error deviation satisfies
$$\delta_\lambda(x)=|f(x) -  {\xi}_{\lambda}(x)|^2 - |\hat e(x)|^2= \frac{\lambda^2(\rho(x) + \alpha(x)) +  {O}(\lambda^3)}{|q_\lambda(x)|^2} < -\frac{\epsilon_1}{2}, \quad \forall x\in\mathcal{E}(\hat{\xi}).$$
 As ${\cal X}$ contains finite points, when $\mathcal{E}(\hat{\xi}) \subsetneq \mathcal{X}$, there exists $\epsilon_2>0$ such that
$$\|\hat{e}\|_{\infty,\mathcal{X}} > |f(x) - \hat{\xi}(x)| + \epsilon_2 \quad \forall x\in\mathcal{X}\setminus\mathcal{E}(\hat{\xi}).$$
Consequently, for sufficiently small $\lambda$, we obtain 
$$|f(x) - \xi_\lambda(x)| < \|\hat{e}\|_{\infty,\mathcal{X}} \quad \forall x\in\mathcal{X},$$
contradicting $\hat{\xi}$ being a local minimax solution to~\eqref{eq:bestf0}.
\end{proof}
 
Analogously to the dual form of the Kolmogorov conditions,   \eqref{eq:RuttanR2} admits the following dual form \eqref{eq:RuttanRD2}  obtained by a linear programming (LP).  
To this end,   for $\mathcal{B}={\cal X}=\{x_j\}_{j=1}^m$,  let  $|{\cal E}(\hat \xi)|=d$ and
\begin{equation}\label{eq:ExtX}
{\cal E}(\hat \xi)=\{z_j\}_{j=1}^d\subseteq {\cal X}.
\end{equation}
As $\hat \xi$ is a local minimax approximant, by \eqref{eq:kolvRD}, we know that $d\ge n+2$.

\begin{theorem}[Ruttan's dual formulation of second-order optimality conditions]\label{thm:RuttanD}
Suppose $\mathcal{B}={\cal X}=\{x_j\}_{j=1}^m$ containing $m\ge N=n_1+n_2+2$ distinct points.   Let $\hat \xi =\frac{\hat p}{\hat q}\in   \scrR_{(n_1,n_2)}$ be an irreducible rational  (local) minimax approximant to $f\in {\bf C}({\cal B})\setminus \scrR_{(n_1,n_2)}$  of \eqref{eq:bestf0} with {defect} $\upsilon(\hat p,\hat q)$. Let $n=n_1+n_2-\upsilon(\hat p,\hat q)$ and ${\cal E}(\hat \xi)=\{z_j\}_{j=1}^d\subseteq {\cal X}.$  Assume $\{\pi_j \}_{j=1}^{n+1}$ is a basis of $\bbP_n$ and $\Pi_n\in \bbC^{d\times (n+1)}$ with $[\Pi_n]_{i,j}=\pi_{j}(z_i)$. Then we have \eqref{eq:kolvR} and \eqref{eq:kolvRD}.  Moreover,  for any $(s,t)\in {\bbS}$ given in \eqref{eq:stconstrant}, it holds that
  \begin{align} \label{eq:RuttanRD2} 
 \left\{\max_{\begin{subarray}{c}\boldsymbol\omega\in \bbR^d,~\boldsymbol\omega\ge 0 \\\boldsymbol{\omega}^{\T}\be=1,~
            E^{\T}   \boldsymbol\omega=\bzs,~\end{subarray}} \boldsymbol \omega^{\T}\boldsymbol\rho(\bz)\right\}\ge 0
\end{align}
where $\be=[1,\dots,1]^{\T}\in \bbR^d$, $\boldsymbol\rho(\bz)=[\rho(z_1),\dots,\rho(z_d)]^{\T}\in \bbR^d$ with $\rho(x)$ given  in  \eqref{eq:varRkrho},  $E=[\Pi^{\tt r},-\Pi^{\tt i}]\in \bbR^{d\times 2(n+1)}$ with 
$$\Pi=D\Pi_n=:\Pi^{\tt r}+{\tt i}\Pi^{\tt i}\in \bbC^{d\times (n+1)},~\Pi^{\tt r}, \Pi^{\tt i}\in \bbR^{d\times (n+1)},$$
  $D=\diag(\tau_1,\dots,\tau_d)$ and $\tau_j=2\overline{\hat e(z_j)\hat q(z_j)}/\hat q(z_j)$. 
\end{theorem}
\begin{proof}
A pair $(a,b) \in \mathbb{P}_{n_1} \times \mathbb{P}_{n_2}$ defines  $c(x) = b(x)\hat{p}(x) - a(x)\hat{q}(x)\in \mathbb{P}_n$, which can be uniquely represented in terms of the basis $\{\pi_j\}_{j=1}^{n+1}$ as
$c(x) = [\pi_1(x), \dots, \pi_{n+1}(x)]  \bc,$
where $\bc \in \mathbb{C}^{n+1}$ is the corresponding coefficient vector. Conversely, for any polynomial in $\mathbb{P}_n$,   $\exists(a,b) \in \mathbb{P}_{n_1} \times \mathbb{P}_{n_2}$ that expresses it in the form of  \eqref{eq:alpha}. Using the notations in this theorem and letting $\bc=\bc^{\tt r}+{\tt i} \bc^{\tt i}$, we have 
$$
[\alpha(z_1),\dots,\alpha(z_d)]^{\T}=\RE( D \Pi_n \bc)=E\left[\begin{array}{c}\bc^{\tt r} \\\bc^{\tt i} \end{array}\right]=:E\wtd \bc~~{\rm with}~\wtd \bc=\left[\begin{array}{c}\bc^{\tt r} \\\bc^{\tt i} \end{array}\right]\in \bbR^{2(n+1)}.
$$

Note that the minimization problem \eqref{eq:RuttanR2} can be written as  
\begin{equation}\label{eq:RuttanR2b}
\min_{\wtd\bc\in\bbR^{2(n+1)}}\max_{z_j\in {\cal E}(\hat \xi)}  (\rho(z_j) + \alpha(z_j)).
\end{equation}
Equivalently, introducing a real variable $\mu\in \bbR$ to bound $$\rho(z_j) + \alpha(z_j)\le \mu,~\forall j=1,\dots,d,$$ we can write \eqref{eq:RuttanR2b} as 
\begin{equation}\label{eq:RuttanR2c}
\min_{\begin{subarray}{c}\wtd\bc\in\bbR^{2(n+1)}\\\rho(z_j) + \alpha(z_j)\le \mu,~\forall j=1,\dots,d
\end{subarray}}\mu.
\end{equation}
With  $E, \boldsymbol\rho(\bz), \alpha(z_j)$ given in the theorem, we   can write $\rho(z_j) + \alpha(z_j)\le \mu,~\forall j=1,\dots,d$ as  
$\mu \be  -E\wtd \bc \ge\boldsymbol\rho(\bz)$. Apparently, \eqref{eq:RuttanR2c} is an LP for the variables $\mu,\wtd \bc$ whose dual LP is given (see e.g., \cite[Chapters 12 and 13]{nowr:2006})  by \eqref{eq:RuttanRD2}, where $\boldsymbol\omega\in \bbR^d$ contains the dual variables $\omega_j ~(j=1,\dots,d)$.
\end{proof}

As a corollary of Theorem \ref{thm:RuttanD}, the following result addresses the case $|{\cal E}(\hat \xi)|=d = n + 2$ and will be used to establish a sufficient and necessary condition for the global minimax approximant.

\begin{corollary}\label{cor:Ruttan2ndD}
Under the assumptions of Theorem \ref{thm:RuttanD}, if $d=n+2$, then there is a unique feasible $\boldsymbol \omega=[\omega_1,\dots,\omega_d]^{\T}>0$ for \eqref{eq:RuttanRD2}, and thus $\boldsymbol \omega^{\T}\boldsymbol\rho(\bz)\ge 0$ where $\boldsymbol\rho(\bz)$, associated with the $(s,t)\in {\bbS}$  in \eqref{eq:stconstrant}, is defined in Theorem \ref{thm:RuttanD}.  Moreover, this  $\boldsymbol \omega$ satisfies the dual form \eqref{eq:kolvRD}, i.e.,
\begin{equation}\label{eq:kolvRDd}
\sum_{j=1}^d \omega_j\overline{\hat e(z_j)\hat q(z_j)}g(z_j)/\hat q(z_j)=0,~~\forall g\in \bbP_{n}.
\end{equation}
\end{corollary}
\begin{proof}
For any feasible $\boldsymbol \omega$ in \eqref{eq:RuttanRD2}, it holds that
$\boldsymbol \omega^{\T}D\Pi_n  =\bzs. $
Since $\Pi_n$ is the generalized Vandermonde matrix associated with the basis $\{\pi_j\}_{j=1}^{n+1}$ of $\mathbb{P}_n$, the condition $d = n+2$ combined with $\boldsymbol \omega^{\T}D\Pi_n  =\bzs$ implies \eqref{eq:kolvRDd}.
As $f \in   {\bf C}({\cal B})\setminus \scrR_{(n_1,n_2)}$, each term $\overline{\hat{e}(z_j)\hat{q}(z_j)} / \hat{q}(z_j)$ is nonzero, ensuring $D$ is nonsingular. Consequently, $\rank(D\Pi_n) = n+1$. The solution space of $\boldsymbol \omega^{\T}D\Pi_n  =\bzs$ is then one-dimensional and, under the   constraints $\boldsymbol \omega^{\T}\be=1$ and $\boldsymbol \omega\ge \bzs$, uniquely determined.
\end{proof}

\newpage
\subsection{A summary of local optimality conditions}\label{subsec:summary}
The following  summarizes the derivation of these optimality conditions for the (local) minimax approximant $\xi$, where $
\delta_{\lambda}(x) = |f(x) - \xi_{\lambda}(x)|^2 - |f(x) - \hat \xi(x)|^2$. 

\vskip2mm

$\bullet$  Kolmogorov's conditions for a general compact ${\cal B}$ (first-order analysis):

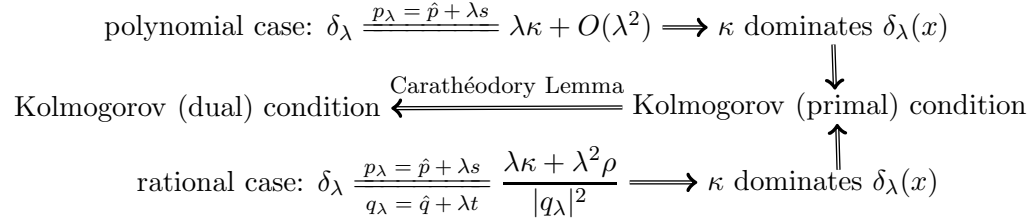
\begin{tikzpicture}[
    node distance=1cm,
    every node/.style={align=center}
]
\node (main) {polynomial case:  $\delta_{\lambda} \xlongequal[\text{}]{\text{$p_{\lambda}=\hat p+\lambda s$}}  {\lambda \kappa +O(\lambda^2)}$};
\node[right=0.6cm of main] (dominates) {$\kappa$ dominates $\delta_{\lambda}(x)$};
\draw[->, double, line width=0.5pt] (main) -- (dominates);
\node[below=0.95cm of main] (main2) {rational case: $\delta_{\lambda} \xlongequal[\text{$q_{\lambda}=\hat q+\lambda t$}]{\text{$p_{\lambda}=\hat p+\lambda s$}} \dfrac{\lambda \kappa +\lambda^2 \rho}{|q_\lambda|^2}$};
\node[right=0.9cm of main2] (dominates2) {$\kappa$ dominates $\delta_{\lambda}(x)$};
\draw[->, double, line width=0.5pt] (main2) -- (dominates2);
\node[below=0.4 cm of dominates, xshift=-0.05cm] (primal) {Kolmogorov (primal) condition};
\draw[->, double, line width=0.5pt] 
    ([yshift= 0.1cm]dominates.south) -- ([xshift=0.07cm, yshift=-0.1cm]primal.north)
    node[pos=0.5, right, font=\footnotesize] {};
\draw[->, double, line width=0.5pt] 
    (dominates2.north) ++(0.205cm,-0.16cm) -- ([xshift=0.1cm, yshift=0.1cm]primal.south)
    node[pos=0.5, right, font=\footnotesize] {};
    
\node[below=0.25cm of main, xshift=-2.4cm] (dual) {Kolmogorov (dual) condition};
\draw[->, double, line width=0.5pt] (primal) -- (dual) 
    node[pos=0.5, above, font=\footnotesize] {Carath\'eodory Lemma};
\end{tikzpicture}

\vskip 2mm
$\bullet$   Ruttan's necessary conditions for a finite discrete set ${\cal B}={\cal X}$ (second-order analysis):

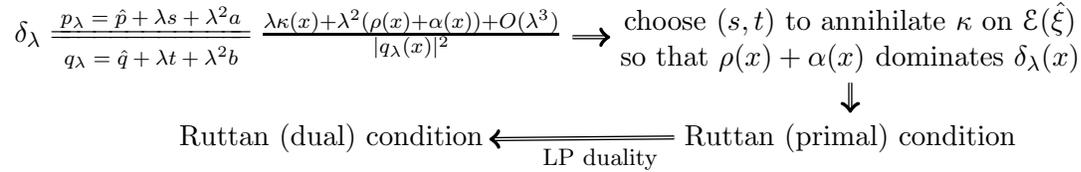
\begin{tikzpicture}[
    node distance=1cm,
    every node/.style={align=center}
]
\node (main) {$\delta_{\lambda} \xlongequal[\text{$q_{\lambda}=\hat q+\lambda t+\lambda^2 b$}]{\text{$p_{\lambda}=\hat p+\lambda s+\lambda^2 a$}}  {\frac{\lambda \kappa(x)+\lambda^2(\rho(x) + \alpha(x)) +  {O}(\lambda^3)}{|q_\lambda(x)|^2}}$};
\node[right=0.5cm of main] (dominates) {choose $(s,t)$ to annihilate $\kappa$ on ${\cal E}(\hat\xi)$ \\so that $\rho(x) + \alpha(x)$ dominates $\delta_{\lambda}(x)$};
\draw[->, double, line width=0.5pt] (main) -- (dominates);
\node[below=0.4cm of dominates] (primal) {Ruttan    (primal) condition};
\draw[->, double, line width=0.5pt] (dominates) -- (primal);
\node[left=2.4cm of primal, xshift= -0.0cm] (dual) {Ruttan (dual) condition};
\draw[->, double, line width=0.5pt] 
    ([yshift=0cm]primal.west) -- ([xshift=2.1cm, yshift=-0.35cm]dual.north)
    node[pos=0.4, below, font=\footnotesize] {LP duality};
\end{tikzpicture}
\section{Sufficient optimality for global approximants}\label{sec:RuttanSuf}
In \cite{rutt:1985}, Ruttan  also introduced a sufficient condition for characterizing global minimax approximants. 
In Section \ref{subsec:ruttanglobal}, we present this condition and demonstrate its applicability to both the continuum and finite discrete sets ${\cal B}$. Furthermore, in Section \ref{subsec:ruttanglobalNec}, we establish that for finite discrete sets ($\mathcal{B}=\mathcal{X}$), this condition becomes necessary when  $\hat \xi$ attains the minimal number of extreme points, i.e., $|\mathcal{E}(\hat{\xi})| = n_1 + n_2 + 2 - \upsilon(\hat{p}, \hat{q})$.

\subsection{Ruttan's sufficient condition}\label{subsec:ruttanglobal}

Central to Ruttan's sufficient condition  is $\rho(x)$ defined in \eqref{eq:varRkrho}. Given an irreducible rational function $\hat{\xi} = \hat{p}/\hat{q} \in \mathscr{R}_{(n_1,n_2)}$ and its associated maximum error $\|\hat{e}\|_{\infty,\mathcal{B}}$, we introduce the auxiliary function
$$h(x; s, t) = \big| f(x)t(x) - s(x) \big|^2 - \big| t(x) \big|^2 \cdot \|\hat{e}\|_{\infty,\mathcal{B}}^2,$$
defined for a parameter pair $(s, t) \in \mathbb{P}_{n_1} \times \mathbb{P}_{n_2}$. It is noted that $h(x; s, t)=\rho(x)~\forall x\in {\cal E}(\hat\xi)$. 

With the polynomial bases of $\bbP_{n_1}$ and $\bbP_{n_2}$ in \eqref{eq:basesP12}, letting 
\begin{equation}\nonumber
s(x)=[\psi_0(x),\dots,\psi_{n_1}(x)]\bs,~\bs\in \bbC^{n_1+1}~t(x)=[\phi_0(x),\dots,\phi_{n_1}(x)]\bt,  ~\bt\in \bbC^{n_2+1},
\end{equation} 
it can be seen \cite{isth:1993} that  
\begin{equation}\nonumber
h(x; s, t)=\left[\begin{array}{c}\bs \\\bt\end{array}\right]^{\HH}H(x)\left[\begin{array}{c}\bs \\\bt\end{array}\right],
\end{equation}
where $H:\bbC\rightarrow \bbC^{N\times N}$ is a rank-one matrix-valued function  with $N=n_1+n_2+2$ given by
\begin{equation} \label{eq:Rutten-rank1}
H(x)=\left[\begin{array}{cc}H_1(x) & H_3(x)  \\H_3^{\HH}(x)  &H_2(x) \end{array}\right]\in \bbC^{N\times N}
\end{equation}
and 
\begin{subequations}\label{eq:Hk}
\begin{align}\label{eq:H1}
[H_1(x)]_{\ell,k}&=\overline{\psi_{\ell-1}(x)}\psi_{k-1}(x)\in \bbC,~~1\le \ell,k\le n_1+1,\\\label{eq:H2}
[H_2(x)]_{\ell,k}&=\left(|f(x)|^2- \|\hat e \|_{\infty,{\cal B}}^2\right)\overline{\phi_{\ell-1}(x)}\phi_{k-1}(x)\in \bbC,~~1\le \ell,k\le n_2+1,\\\label{eq:H3}
[H_3(x)]_{\ell,k}&=-f(x)\overline{\psi_{\ell-1}(x)}\phi_{k-1}(x)\in \bbC,~~1\le \ell\le n_1+1,~~1\le k\le n_2+1.
\end{align}
\end{subequations}
One can verify that if $\{\wtd \psi_j\}_{j=0}^{n_1}$ and $\{\wtd \phi_j\}_{j=0}^{n_2}$ are new bases for $\bbP_{n_1}$ and $\bbP_{n_2}$, respectively, with the transformation matrices $T_1\in \bbC^{(n_1+1)\times (n_1+1)}$ and $T_2\in \bbC^{(n_2+1)\times (n_2+1)}$ satisfying 
\begin{equation}\label{eq:transT12}
[\psi_0(x),\dots,\psi_{n_1}(x)]=[\wtd\psi_0(x),\dots,\wtd\psi_{n_1}(x)]T_1,~[\phi_0(x),\dots,\phi_{n_2}(x)]=[\wtd\phi_0(x),\dots,\wtd\phi_{n_2}(x)]T_2,
\end{equation}
then the new rank-one matrix $\wtd H(x)$ in \eqref{eq:Rutten-rank1} associated with the $\{\wtd \psi_j\}_{j=0}^{n_1}$ and $\{\wtd \phi_j\}_{j=0}^{n_2}$  is
$$
\wtd H(x)=\left[\begin{array}{cc}T_1^{\HH}  &  \\   &T_2^{\HH} \end{array}\right]H(x)\left[\begin{array}{cc}T_1  &  \\   &T_2 \end{array}\right].
$$
This implies that the definiteness of $H(x)$ is invariant under the choice of bases for $\bbP_{n_1}$ and $\bbP_{n_2}$.

Another form   in Ruttan \cite{rutt:1985} to express $h(x;s,t)$ is to use a basis $\{(\theta_j,\vartheta_j)\}_{j=1}^N$ for the product subspace $\bbP_{n_1}\times \bbP_{n_2}$. If  $\bbP_{n_1}$ and $\bbP_{n_2}$ have bases in \eqref{eq:basesP12}, one choice of $\{(\theta_j,\vartheta_j)\}_{j=1}^N$ can be
$
\left[(\psi_0,0), \dots,(\psi_{n_1},0), (0, \phi_0), \dots,(0,\phi_{n_2}) \right].
$
Using $\{(\theta_j,\vartheta_j)\}_{j=1}^N$, a pair $(s,t)\in \bbP_{n_1}\times \bbP_{n_2}$ can be expressed by 
\begin{equation}\nonumber
s(x)=[\theta_1(x),\dots, \theta_N(x)]\bv, ~t(x)=[\vartheta_1(x),\dots, \vartheta_N(x)]\bv,~\bv\in \bbC^{N},
\end{equation}
and the function $h(x;s,t)$ can be given by 
$
h(x;s,t)=\bv^{\HH}G(x)\bv,
$
where $G:\bbC\rightarrow \bbC^{N\times N}$ is a rank-one matrix-valued function with the entry 
\begin{equation}\label{eq:G}
[G(x)]_{\ell,k}=\overline{(f(x)\vartheta_\ell(x)-\theta_\ell(x))}(f(x)\vartheta_\ell(x)-\theta_\ell(x))-\|\hat e \|_{\infty,{\cal B}}^2 \cdot\overline{\vartheta_\ell(x)}\vartheta_\ell(x).
\end{equation}
Analogously, if both $\{(\theta_j,  \vartheta_j)\}_{j=1}^{N}$  and $\{(\wtd \theta_j,\wtd \vartheta_j)\}_{j=1}^{N}$ are bases of $\bbP_{n_1}\times \bbP_{n_2}$ with the transformation $T \in \bbC^{N\times N}$ similar to \eqref{eq:transT12}, then the associated rank-one matrix $\wtd G(x)=T^{\HH}G(x) T$. This implies that the definiteness of $G(x)$ is invariant under the choice of basis for $\bbP_{n_1} \times \bbP_{n_2}$.

We now present  Ruttan's sufficient condition \cite[{Theorem} 2.1]{rutt:1985} for a global minimax approximant $\hat \xi =\hat p /\hat q$.

\begin{theorem}{\rm (\cite[{Theorem} 2.1]{rutt:1985})}\label{thm:RuttanSufOpt}
Suppose $\mathcal{B} \subset \mathbb{C}$ is a  compact set containing at least $N=n_1+n_2+2$ distinct points and $\hat\xi=\hat p /\hat q \in  \scrR_{(n_1,n_2)}$ is an irreducible rational {approximant} of $f\in {\bf C}({\cal B})\setminus \scrR_{(n_1,n_2)}$. If there exist points $\{z_j\}_{j=1}^r\subseteq {\cal E}(\hat\xi),~r\ge 1$ and positive real constants $\{\omega_j\}_{j=1}^r$ such that $\sum_{j=1}^r \omega_j=1$ and the Hermitian matrix\footnote{Condition \eqref{eq:RuttanHc} that the Hermitian matrix $H$ is positive semi-definite can equivalently be  $G=\sum_{j=1}^r \omega_j G(z_j)\succeq 0$, where $G(x)$ is defined in \eqref{eq:G}.}
\begin{equation}\label{eq:RuttanHc}
H =\sum_{j=1}^r \omega_j H(z_j)\succeq 0,
\end{equation}
where $H(x)$ is given by \eqref{eq:Rutten-rank1}, then  $\hat\xi$ is a global minimax approximant for \eqref{eq:bestf0}. 
In that case $\{z_j\}_{j=1}^r$ and $\{\omega_j\}_{j=1}^r$ satisfy the local (dual form) Kolmogorov condition \eqref{eq:kolvRD}.
\end{theorem}
The proof of \cite[Theorem 2.1]{rutt:1985} implicitly assumes that \eqref{eq:bestf0} admits a solution—that is, the infimum $\eta_{\mathcal{B}}$ is attainable. However, condition \eqref{eq:RuttanHc} inherently ensures the existence of a minimax approximant, which renders the result applicable even when $\mathcal{B} = \mathcal{X}$ is a finite discrete set. To see this,  suppose, by contrast, that $\hat\xi$ is not a global solution to \eqref{eq:bestf0}. Then either $\eta_{\mathcal{B}}$ is not attainable or  is attainable but $\hat\xi$ is not a solution. In either case, it must hold that $\|\hat e\|_{\infty,{\cal B}}>\eta_{\mathcal{B}}$, and we can find an irreducible $\wtd \xi(x)=\wtd p(x)/\wtd q(x)\in \scrR_{(n_1,n_2)}$ with $\wtd q(x)\ne 0~(x\in {\cal B})$ satisfying 
 \begin{equation}\label{eq:etawtd}
 \eta_{\mathcal{B}}\le \|\wtd e\|_{\infty, {\cal B}}=\max_{x\in {\cal B}} |f(x)-\wtd \xi(x)|<\|\hat e\|_{\infty,{\cal B}}.
 \end{equation}
 Indeed, if the infimum of \eqref{eq:bestf0} is unattainable and $\|\hat e\|_{\infty,{\cal B}}>\eta_{\cal B}$, then by the definition of infimum \eqref{eq:bestf0},  we can find such a $\wtd \xi$ satisfying  \eqref{eq:etawtd}; otherwise, if   \eqref{eq:bestf0} is attainable but $\hat\xi$ is not a global solution, then simply take $\wtd \xi$ as any global minimax approximation. 
 
 Let $\wtd \bs\in \bbC^{n_1+1}$ and $\wtd \bt\in \bbC^{n_2+1}$ be the associated coefficient vectors of $ \wtd p(x) $ and $\wtd q(x)$ in the basis of \eqref{eq:basesP12}. By \eqref{eq:RuttanHc}, it follows that
 \begin{align*}
 0&\le \sum_{j=1}^r \omega_j\left[\begin{array}{c}\wtd\bs \\\wtd\bt\end{array}\right]^{\HH}H(z_j)\left[\begin{array}{c}\wtd\bs \\\wtd\bt\end{array}\right]=\sum_{j=1}^r \omega_j h(z_j; \wtd p,\wtd q)\\
 &=\sum_{j=1}^r \omega_j\left(\big| f(z_jx)\wtd q(z_j) - \wtd p(z_j) \big|^2 - \big| \wtd q(z_j) \big|^2 \cdot \|\hat{e}\|_{\infty,\mathcal{B}}^2\right)\\
 &=\sum_{j=1}^r \frac{\omega_j}{\big| \wtd q(z_j) \big|^2}\left(\big| f(z_j)  - \wtd \xi(z_j) \big|^2 - \|\hat{e}\|_{\infty,\mathcal{B}}^2\right).
 \end{align*}
 As $r\ge 1$ and $\omega_j>0$, there is a $j$ so that $\big| f(z_j)  - \wtd \xi(z_j) \big| \ge  \|\hat{e}\|_{\infty,\mathcal{B}}$,  and therefore, 
 $$
\|\wtd e\|_{\infty, {\cal B}}=\max_{x \in {\cal B}} \|f(x)-\wtd \xi(x)\| \ge\big| f(z_j)  - \wtd \xi(z_j) \big| \ge \|\hat  e\|_{\infty, {\cal B}},
 $$
 which contradicts    \eqref{eq:etawtd}. This proves that $\hat \xi$ achieves the infimum in  \eqref{eq:bestf0} and thus is a global minimax approximant. The rest part of the proof follows \cite[Theorem 2.1]{rutt:1985}.

Ruttan's characterization of the global approximant has limited practical utility for three key reasons:
(i) it requires a priori computation of an irreducible $\hat{\xi}$ to verify \eqref{eq:RuttanHc};
(ii) it necessitates identifying a subset $\{z_j\}_{j=1}^r \subseteq  {\cal E}(\hat \xi)$ along with associated positive weights $\{\omega_j\}_{j=1}^r$; and
(iii) it imposes the condition that the matrix $H$ \eqref{eq:RuttanHc} be positive semi-definite. Fortunately, for finite discrete sets $\mathcal{B}$, this characterization can be efficiently verified numerically in a recently developed method, the {\tt d-Lawson} iteration proposed in \cite{zhyy:2025,zhha:2025}. We will explore this connection in detail in Section \ref{sec:dlawsonFound}.
 
\subsection{When Ruttan's sufficient condition becomes necessary}\label{subsec:ruttanglobalNec}

In \cite{rutt:1985}, for any compact set ${\cal B} \subseteq \bbC$ and $f \in {\bf C}({\cal B})$, Ruttan asserted that the sufficient condition in Theorem \ref{thm:RuttanSufOpt} is also necessary when   $|{\cal E}(\hat\xi)|=n+2$.   By Theorem \ref{thm:kolR}, we know that $n+2$ is the minimal cardinality of $\mathcal{E}(\hat{\xi})$ for $\hat\xi$ to be a local minimax approximant. However, \cite{this:1993} later showed that this claim does not hold for arbitrary compact subsets $\mathcal{B} \subset \mathbb{C}$. Notably, \cite{this:1993} identified a special case—where $\mathcal{B} = [a,b] \subseteq \bbR$—and proved that when $|{\cal E}(\hat\xi)|=n+2$, Ruttan's sufficient condition becomes necessary for approximating  $f \in {\bf C}([a,b])$ over real normal rational functions (\cite[Theorem 5]{this:1993}). The key contribution of this subsection is generalizing this necessity result to arbitrary finite discrete sets ${\cal B} \subset \bbC$.
   
\begin{theorem}\label{thm:Ruttansufnec}
Suppose $\mathcal{B}={\cal X}=\{x_j\}_{j=1}^m$ containing $m\ge N=n_1+n_2+2$ distinct points and $f\in {\bf C}({\cal B})\setminus \scrR_{(n_1,n_2)}$. Let  $\hat\xi=\hat p/\hat q\in   \scrR_{(n_1,n_2)}$ be   irreducible and  ${\cal E}(\hat\xi) =\{z_j\}_{j=1}^d$ with $d=n_1+n_2-\upsilon(\hat p,\hat q)+2$. Then $\hat\xi$ is a global minimax approximant of \eqref{eq:bestf0} if and only if  there exist $d$    positive real constants $\{\omega_j\}_{j=1}^d$ such that\footnote{Equivalently, the condition $H =\sum_{j=1}^d \omega_j H(z_j)\succeq 0$ can be  $G=\sum_{j=1}^r \omega_j G(z_j)\succeq 0$, where $G(x)$ is defined in \eqref{eq:G}.} $\sum_{j=1}^d \omega_j=1$ and  $H =\sum_{j=1}^d \omega_j H(z_j)\succeq 0$, where $H(x)$ is given in \eqref{eq:Rutten-rank1}. 
\end{theorem}
\begin{proof}
The proof is the same as \cite[Theorem 2.2]{rutt:1985}; we include a sketch for completeness. 

It suffices to establish necessity. Let $\hat{\xi}$ be a global minimax approximant. Under the assumptions, we can construct a basis $\{(\theta_j,\vartheta_j)\}_{j=1}^N$ for $\bbP_{n_1}\times \bbP_{n_2}$ \cite[p. 292-293]{rutt:1985} so that 
\begin{equation}\label{eq:realthetavartheta}
\RE\left(\overline{\hat e(x) \hat q(x)}(\vartheta_j(x)\hat p(x)-\theta_j(x)\hat q(x))/\hat q(x)\right)=0,~~\forall 1\le j\le N, \forall x\in {\cal E}(\hat \xi).
\end{equation}
Furthermore, by Corollary \ref{cor:Ruttan2ndD}, there exists a unique $\boldsymbol \omega=[\omega_1,\dots,\omega_d]^{\T}> \boldsymbol{0}$  such that $\boldsymbol \omega^{\T} \boldsymbol{\rho}(\bz) \geq 0$ and $\sum_{j=1}^d\omega_j=1$, where $\boldsymbol{\rho}(\bz)$ defined in \eqref{eq:varRkrho} corresponds to the $(s,t)$ pair specified in \eqref{eq:stconstrant} and $\bz$ contains entries $z_j\in {\cal E}(\hat \xi)$. Using this $\boldsymbol \omega$ and all points $z_j \in \mathcal{E}(\hat{\xi})$, we define
$G = \sum_{j=1}^d \omega_j G(z_j).$
We prove  the positive semi-definiteness of $G$ by contradiction. Assume $G\not \succeq 0$. As $G$ is Hermitian, there is a real vector $\bv=[v_1,\dots,v_N]^{\T}\in \bbR^N$ so that $\bv^{\T}G\bv<0$. Define 
$$(s,t)=\sum_{j=1}^N(\theta_j,\vartheta_j)v_j\in \bbP_{n_1}\times \bbP_{n_2} ~\mbox{and}~g(x)=\hat p(x)t(x)-\hat q(x)s(x).$$ By \eqref{eq:realthetavartheta} and $v_j\in \bbR$, it then follows that 
$$
 {2 \RE\left(\overline{\hat e(x)\hat q(x)}g(x)/\hat q(x)\right)}=0,~\forall x\in\mathcal{E}(\hat{\xi})\Longrightarrow (s,t)\in \bbS,
$$
giving $\boldsymbol \omega^{\T} \boldsymbol{\rho}(\bz) \geq 0$ by Corollary \ref{cor:Ruttan2ndD}. But on the other hand, since $h(x; s, t)=\rho(x)~\forall x\in {\cal E}(\hat\xi)$ and $h(x;s,t)=\bv^{\HH}G(x)\bv$, it holds that
$$0>\bv^{\T}G\bv=\sum_{j=1}^d \omega_j h(z_j; s,t)=\sum_{j=1}^d \omega_j \rho(z_j)=\boldsymbol \omega^{\T} \boldsymbol{\rho}(\bz),$$
a contradiction with  $\boldsymbol \omega^{\T} \boldsymbol{\rho}(\bz) \geq 0$. This shows that $G\succeq 0.$ 
\end{proof}

\section{Foundations for the  {\tt d-Lawson} iteration}\label{sec:dlawsonFound}
In this section, we assume that ${\cal B}={\cal X}=\{x_j\}_{j=1}^m \subseteq \Omega$ is a finite discrete set with $m\ge n_1+n_2+2$.  With the bases of $\bbP_{n_1}$ and $\bbP_{n_2}$ in \eqref{eq:basesP12}, we write   $\xi=p/q\in \scrR_{(n_1,n_2)}$ as 
\begin{equation} \label{eq:paramt_ab}
\xi(x)=\frac{p(x)}{q(x)}=\frac{[\psi_0(x),\dots,\psi_{n_1}(x)] \ba}{[\phi_0(x),\dots,\phi_{n_2}(x)] \bb},~~\mbox{for ~some~} \ba\in \bbC^{n_1+1}, ~\bb\in \bbC^{n_2+1},
\end{equation}  
and define coefficient matrices 
\begin{equation}\nonumber 
\Psi=\Psi(x_1,\dots,x_m;n_1):=\left[\begin{array}{cccc}\psi_0(x_1) & \psi_1(x_1) & \cdots & \psi_{n_1}(x_1) \\\psi_0(x_2) & \psi_1(x_2) & \cdots & \psi_{n_1}(x_2)  \\ \vdots & \cdots & \cdots& \vdots  \\\psi_0(x_m) & \psi_1(x_m) & \cdots & \psi_{n_1}(x_m) \end{array}\right],~\Psi_{i,j}=\psi_{j-1}(x_i);
\end{equation} 
analogously, we let $\Phi=\Phi(x_1,\dots,x_m;n_2)=[\phi_{j-1}(x_i)] \in \bbC^{m\times (n_2+1)}$.

In \cite{zhyy:2025}, a max-min type dual  formulation of the original minimax approximation \eqref{eq:bestf0} is developed and is defined by
\begin{equation}\label{eq:dual}
\max_{\bw\in {\cal S}} d(\bw),
\end{equation} 
with the dual function $d(\bw)$  given by ($\bw$ is the dual variable) 
\begin{align}\nonumber
d(\bw)&=\min_{\begin{subarray}{c}p\in \bbP_{n_1},~q\in \bbP_{n_2}\\
            \sum_{j=1}^m w_j |q(x_j)|^2=1\end{subarray}}\sum_{j=1}^m w_j |f_j q(x_j)-p(x_j)|^2\\\label{eq:rat-d-compt}
            &=\min_{\begin{subarray}{c}  \ba\in \bbC^{n_1+1},~ \bb\in \bbC^{n_2+1}\\
            \|\sqrt{W}\Phi \bb\|_2 =1\end{subarray}}\left\|\sqrt{W}[-\Psi,F\Phi]  \left[\begin{array}{c}\ba \\\bb\end{array}\right]\right\|_2^2,
\end{align}
where $F=\diag(f(x_1),\dots,f(x_m))$, $W=\diag(\bw)$ and the  probability simplex 
\begin{equation}\nonumber 
{\cal S}\triangleq\{\bw=[w_1,\dots,w_m]^{\T}\in \bbR^m: \bw\ge 0 ~{\rm and } ~\bw^{\T}\be=1\},~~\be=[1,\dots,1]^{\T}.
\end{equation}

For any $\bw\in {\cal S}$, we introduce the $\bw$-inner (positive semi-definite) product defined by $\langle\by,\bz \rangle_{\bw}=\by^{\HH}W\bz$ and $\|\by\|_{\bw}=\sqrt{\by^{\HH}W\by}$. The following proposition provides the optimality conditions for the pair $(\ba(\bw), \bb(\bw))$ to be the solution of \eqref{eq:rat-d-compt} at given $\bw\in {\cal S}$. 
\begin{proposition}{\rm (\cite[Corollary 2.7]{zhha:2025})}\label{prop:optimality}
Let $ \bp=\Psi \ba(\bw)\in\bbC^m$ and $\bq=\Phi \bb(\bw)\in\bbC^m$ be from the solution of \eqref{eq:rat-d-compt} with the weight vector $\bw\in {\cal S}$. Then 
\begin{equation}\label{eq:optimalitytwo}
F\bq-\bp\perp_{\bw}{\rm span}(\Psi),~~F^{\HH}(F\bq-\bp)-d(\bw)\bq\perp_{\bw}{\rm span}(\Phi).
\end{equation}
\end{proposition}
We shall show in Section \ref{subsec:foundLawson} that \eqref{eq:optimalitytwo} is just the Kolmogorov dual criteria \eqref{eq:kolvRD} when $\what \bw$ is the global solution of \eqref{eq:dual}. 

For solving the dual problem \eqref{eq:dual}, \cite{zhyy:2025} proposes the {\tt d-Lawson} iteration (Algorithm \ref{alg:Lawson}) whose convergence has been recently developed in \cite[Theorem 5.2]{zhha:2025}.

\begin{algorithm}[thb!!!]
\caption{The {\tt d-Lawson} iteration \cite{zhyy:2025} for \eqref{eq:bestf0} with ${\cal B}=\{x_j\}_{j=1}^m$} \label{alg:Lawson}
\begin{algorithmic}[1]
\renewcommand{\algorithmicrequire}{\textbf{Input:}}
\renewcommand{\algorithmicensure}{\textbf{Output:}}
\REQUIRE Given samples $\{(x_j,f(x_j))\}_{j=1}^m$ ($n_1+n_2+2\le m$) with $x_j\in \Omega$, a relative tolerance for the strong duality $\epsilon_r>0$, the maximum number $k_{\rm maxit}$ of iterations; 
        \smallskip
\STATE  (Initialization) Let $k=0$; choose $0<\bw^{(0)}\in {\cal S}$  and a  tolerance $\epsilon_w$ for the weights;
\STATE (Filtering) Remove nodes $x_i$ with $w_i^{(k)}<\epsilon_w$;
\STATE Compute  $d(\bw^{(k)})$ and the associated vector $\xi^{(k)}(\bx)={[\Psi \ba(\bw^{(k)})]./[\Phi \bb(\bw^{(k)})]}$  from \eqref{eq:rat-d-compt};
  
\STATE ({Stopping} rule)  Stop either if $k\ge k_{\rm maxit}$ or 
\begin{equation}\nonumber 
\epsilon(\bw^{(k)}):=\left|\frac{\sqrt{d(\bw^{(k)})}-\zeta(\xi^{(k)})}{\zeta(\xi^{(k)})}\right|<\epsilon_r,~~{\rm where}~~\zeta(\xi^{(k)})=\|f-\xi^{(k)}\|_{\infty,{\cal X}};
\end{equation} 
 
\STATE (Updating weights) Update the weight vector $\bw^{(k+1)}$ according to 
\begin{equation}\nonumber
w_j^{(k+1)}=\frac{w_j^{(k)}\left|f(x_j)-\xi^{(k)}(x_j)\right|^{\beta}}{\sum_{i}w_i^{(k)}\left|f(x_i)-\xi^{(k)}(x_i)\right|^{\beta}},~~\forall j,
\end{equation} 
with the Lawson exponent $\beta>0$, and {go to} Step 2 with $k=k+1$.
\end{algorithmic}
\end{algorithm}  
\subsection{Strong duality and Ruttan's sufficient global optimality}\label{subsec:RuttanStrong}

Let $\what \bw\in {\cal S}$ be any global maximizer of the dual problem \eqref{eq:dual} with the corresponding pair $(\ba(\what \bw),\bb(\what \bw))$ from \eqref{eq:rat-d-compt} at $\what \bw$. Let $\hat \xi$ be the irreducible representation of the rational function given \eqref{eq:paramt_ab} with coefficient vectors  $\ba(\what \bw)$ and $\bb(\what \bw)$, and
$
\|\hat e\|_{\infty,{\cal X}}=\max_{x_j\in {\cal X}}|f(x_j)-\hat \xi(x_j)|.
$
The strong duality is stated as \cite{zhyy:2025}
\begin{equation}\label{eq:strongdualityRuttan}
\sqrt{d(\what \bw)}=\|\hat e\|_{\infty,{\cal X}}.
\end{equation}

\begin{theorem}{\rm (\cite[Theorem 4.2]{zhyy:2025})}\label{thm:strongdualityeqvRuttan}
Suppose $\mathcal{B}={\cal X}=\{x_j\}_{j=1}^m$ containing $m\ge N=n_1+n_2+2$ distinct points. Ruttan's sufficient condition \eqref{eq:RuttanHc} is satisfied if and only if   strong duality \eqref{eq:strongdualityRuttan} holds. Thus, the global     minimax approximant $\hat\xi\in   \scrR_{(n_1,n_2)}$ of \eqref{eq:bestf0} can be computed from  \eqref{eq:dual} whenever Ruttan's sufficient condition \eqref{eq:RuttanHc} is satisfied.
\end{theorem}

To understand this connection, we first point out that the optimality condition for the pair $(\ba(\what{\bw}), \bb(\what{\bw}))$ to be a minimizer of \eqref{eq:rat-d-compt} at $\what{\bw}=[\what w_1,\dots,\what w_m]^{\T}\in {\cal S}$ is 
\begin{equation}\label{eq:GEP}
[A_{\what \bw} -d(\what \bw) B_{\what \bw}]\left[\begin{array}{c}\ba(\what{\bw}) \\\bb(\what{\bw})\end{array}\right]=\bzs,
\end{equation}
and the matrix $H_{\what \bw}= A_{\what \bw} -d(\what \bw) B_{\what \bw}$ is  positive semi-definite  \cite[Proposition 3.1]{zhyy:2025}, where  
\begin{align*}
 H_{\what \bw}= A_{\what \bw} -d(\what \bw) B_{\what \bw}= \left[\begin{array}{cc}\Psi^{\HH}W\Psi & -\Psi^{\HH}FW\Phi \\-\Phi^{\HH} WF^{\HH}\Psi & \Phi^{\HH}F^{\HH}WF\Phi-d(\what\bw)\Phi^{\HH}W\Phi\end{array}\right].\end{align*}
Equivalently, the optimality condition can be expressed as follows:
the pair $\left( d(\widehat{\bw}), \begin{bmatrix} \ba(\widehat{\bw}) \\ \bb(\widehat{\bw}) \end{bmatrix} \right)$
is an eigenpair of the matrix pencil $(A_{\widehat{\bw}}, B_{\widehat{\bw}})$ corresponding to its smallest eigenvalue $d(\widehat{\bw})$.
In particular, condition \eqref{eq:GEP} holds if and only if \eqref{eq:optimalitytwo} is satisfied.
 
Now, if \eqref{eq:strongdualityRuttan} is true, by the definitions of $H_k(x)$ for $k=1,2,3$ in \eqref{eq:Rutten-rank1} and \eqref{eq:Hk},  it holds that $$[\Psi^{\HH}\be_j\be_j^{\T}\Psi]_{\ell,k}=\be_{\ell}^{\T}\Psi^{\HH}\be_j\be_j^{\T}\Psi\be_k=\overline{\psi_{\ell-1}(x_j)}\psi_{k-1}(x_j),$$ and thus
 \begin{subequations}\nonumber
 \begin{align}\label{eq:HwRuttan1}
 &\Psi^{\HH}W\Psi=\sum_{j=1}^m \what w_j  \Psi^{\HH}\be_j\be_j^{\T}\Psi=\sum_{j=1}^m \what w_j   H_1(x_j),~ -\Psi^{\HH}FW\Phi=\sum_{j=1}^m \what w_j   H_3(x_j),\\\label{eq:HwRuttan2}
 &\Phi^{\HH}F^{\HH}WF\Phi-d(\what \bw)\Phi^{\HH}W\Phi=\sum_{j=1}^m \what w_j   H_2(x_j);
 \end{align}
 \end{subequations}
 consequently,  by \eqref{eq:Rutten-rank1} and \eqref{eq:Hk}, 
 $$0 \preceq H_{\what \bw}=\sum_{j=1}^m \what w_j  H(x_j)=\sum_{j: x_j\in{\cal E}(\hat \xi)} \what w_j H(x_j),$$
 where we have used the fact \cite[Theorem 2.4]{zhyy:2025}: $\what w_j=0 ~\forall x_j\not \in {\cal E}(\hat \xi)$. Thus,   as  $|{\cal E}(\hat \xi)|\ge 1$  and $\what\bw\in {\cal S}$, the matrix $ H_{\what \bw}$ serves as the positive semi-definite Hermitian matrix $H$ in \eqref{eq:RuttanHc}. That is,  Ruttan's sufficient condition \eqref{eq:RuttanHc} is fulfilled under  \eqref{eq:strongdualityRuttan}.

Unlike the impracticality of verifying Ruttan's sufficient condition \eqref{eq:RuttanHc}  in Theorem \ref{thm:RuttanSufOpt}, the equivalence condition \eqref{eq:strongdualityRuttan} provides a numerically viable alternative. Specifically, we can solve the dual problem \eqref{eq:dual} using, for example, the {\tt d-Lawson} iteration, and validate \eqref{eq:strongdualityRuttan} numerically at the obtained solution. Furthermore, the duality gap serves as a practical termination criterion (see Step 4 in Algorithm \ref{alg:Lawson}) for the {\tt d-Lawson} iteration.

\subsection{{\tt d-Lawson} meets Kolmogorov's condition under strong duality}\label{subsec:foundLawson}
In this subsection, we show that under strong duality \eqref{eq:strongdualityRuttan}, if $\what \bw\in {\cal S}$ is the global maximizer of the dual problem \eqref{eq:dual}, the optimality condition in  \eqref{eq:optimalitytwo} is indeed the   Kolmogorov dual criteria \eqref{eq:kolvRD}. To see this, let $\what \bp=[\hat p(x_1),\dots,\hat p(x_m)]^{\T}$ and $\what \bq=[\hat q(x_1),\dots,\hat q(x_m)]^{\T}$ where $(\hat p,\hat q)$ is a pair of polynomials associated with $\what\bw$. 

Consider $F\what\bq-\what\bp\perp_{\what\bw}{\rm span}(\Psi)$ first.  Based on the fact \cite[Theorem 2.4]{zhyy:2025} that $\what w_j=0 ~\forall x_j\not \in {\cal E}(\hat \xi)$, one can write it as 
$$
\sum_{j=1}^m \what w_j \overline{(f(x_j)\hat q(x_j)-\hat p(x_j))}\psi_i(x_j)=\sum_{j: x_j\in {\cal E}(\hat \xi)}  \what w_j \overline{(\hat e(x_j)\hat q(x_j))}\psi_i(x_j)=0,~i = 0,\dots,n_1,
$$
where we assume $\hat{q}(x_j) \neq 0$, since otherwise the boundedness of $\hat{\xi}(x_j)$ would imply $\hat{p}(x_j) = 0$, leaving the sum unchanged.  The above condition is equivalent to \eqref{eq:kolvRD} by choosing $g(x)=\hat q(x)\psi_i(x)\in \bbP_{n}$.

Now, consider the second condition $F^{\HH}(F\what\bq-\what \bp)-d(\what\bw)\what \bq\perp_{\what\bw}{\rm span}(\Phi)$. By strong duality $\sqrt{d(\what \bw)}=\|\hat e\|_{\infty,{\cal X}}=|\hat e(x)|,~\forall x\in {\cal E}(\hat \xi)$, one can write it as ($\forall i=0,\dots,n_2$)
\begin{align*}
0&=\sum_{j=1}^m \what w_j \left(f(x_j)\overline{(f(x_j)\hat q(x_j)-\hat p(x_j))}-\|\hat e\|_{\infty,{\cal X}}^2 \cdot \overline{\hat q(x_j)}\right)\phi_i(x_j) \\
&=\sum_{j: x_j\in {\cal E}(\hat \xi)}  \what w_j \left(f(x_j)\overline{(f(x_j)\hat q(x_j)-\hat p(x_j))}- |\hat e(x_j)|^2 \cdot \overline{\hat q(x_j)}\right)\phi_i(x_j)\\
&=\sum_{j: x_j\in {\cal E}(\hat \xi)}  \what w_j \left(f(x_j)\overline{(\hat e(x_j) \hat q(x_j))}- \overline{(\hat e(x_j) \hat q(x_j))}\cdot  {\hat e(x_j)}\right)\phi_i(x_j)\\
&=\sum_{j: x_j\in {\cal E}(\hat \xi)}  \what w_j \left( \overline{(\hat e(x_j) \hat q(x_j))} \hat p(x_j)/\hat q(x_j) \right)\phi_i(x_j),
\end{align*}
which corresponds to \eqref{eq:kolvRD} with $g(x)=\hat p(x)\phi_i(x)\in \bbP_{n}$. 

From the above analysis, we find that, under strong duality \eqref{eq:strongdualityRuttan}, the solution computed by {\tt d-Lawson} automatically identify the associated the support extreme points $x_j$ and the corresponding weights $\omega_j\ge 0$ in the Kolmogorov dual criteria \eqref{eq:kolvRD}.
\section{From discrete boundary to continuum minimax solution}\label{sec:boundcont}

In many scenarios, it is desirable to compute the rational minimax approximant of a function $f$ defined on a continuum $\Omega \subset \bbC$. Common examples include closed intervals on the real or imaginary axis, the unit circle, or simply connected domains bounded by a Jordan curve $\Gamma = \partial \Omega$. 
For such cases, we assume $f \in \mathbf{C}_{A}(\Omega)$, meaning $f$ is continuous on $\Omega \cup \Gamma$ and analytic in $\mathrm{int}(\Omega)$ whenever the interior is non-empty. 
The study of rational minimax approximation in these settings holds both mathematical significance (see, e.g., \cite{bern:1938,saff:995,sast:1997,stah:1993,stah:1994,stah:2003,tref:1981,regu:1983}) and substantial practical value across applications (cf. \cite{drnt:2024,gamp:2018,guti:2017,limp:2022,nafr:2016,saem:2020,tapo:2014,tref:2019a}).

In contrast to rational approximations, the theory of polynomial minimax approximations is well-established and theoretically profound. A key distinction lies in the existence of necessary and sufficient conditions for arbitrary compact sets $\mathcal{B} \subset \bbR$ or $\bbC$, making polynomial approximations uniquely tractable. This comprehensive theoretical framework has enabled the development of efficient numerical methods for various applications. Notable among these is the classical Remez algorithm \cite{reme:1934,reme:1934b}, with modern implementations discussed in \cite{drht:2014,patr:2009}. Furthermore, extensive research has characterized the density and distribution of extreme points in minimax polynomial approximation (see e.g., \cite{blat:1989,kroo:1981,krsa:1988,lore:1984}), and  asymptotic behavior of zeros of minimax polynomials  (e.g., \cite{blis:1987,grsa:1988}).  
 
 The analysis becomes considerably more challenging for rational minimax approximations. As noted in Section \ref{sec:intro}, the optimality conditions governing both local and global rational minimax approximations are significantly less developed compared to their linear counterparts \cite{chen:1982,mein:1967,rice:1969,rish:1961,shap:1971,sing:1970}. Unlike linear approximations, rational minimax problems exhibit pronounced domain dependence, and no universally applicable necessary and sufficient conditions exist for characterizing global solutions \cite{gutk:1983,sing:2006}. Moreover, fundamental differences emerge in the optimality criteria across (a) compact continuum domains versus discrete point sets, and (b) real-valued versus complex-valued function approximations.

To frame our subsequent discussion, let $\Omega \subset \mathbb{C}$ be a Jordan domain bounded by a rectifiable Jordan curve $\Gamma = \partial\Omega$, and consider $f\in {\bf C}_{A}(\Omega)\setminus \scrR_{(n_1,n_2)}$. Clearly, obtaining an explicit formulation for the minimax approximant $\hat{\xi}$ is generally intractable for arbitrary $\Omega$. Instead, a practical computational approach involves solving a discrete minimax approximation problem over a judiciously selected finite set $\mathcal{X} = \{x_j\}_{j=1}^m \subset \Gamma$. One theoretical foundation for this scheme is the maximum module principle: the extreme points associated with  $\hat \xi$ must be in the boundary $\Gamma$. While this discrete approach proves effective for polynomial approximation when $\mathcal{X}$ contains the extreme points in an extremal signature \cite{rish:1961} (see also the generalized de la Vall\'ee Poussin Theorem \cite[Theorem 2.3.3.1]{shap:1971}) or satisfies appropriate density conditions on $\Gamma$ \cite[Chapter 3]{chen:1982}, the situation becomes more nuanced for rational approximations.

To illustrate the intrinsic features of rational minimax approximation, we consider an example from Gutknecht and Trefethen \cite{regu:1983} where $\Omega={\cal D}$ denotes the unit disk $\mathcal{D} = \{x \in \mathbb{C} : |x| \leq 1\}$ and $f(x)=x^3+x$. This example demonstrates the nonuniqueness of best rational approximants of type $(0,1)$. Numerical results from \cite{this:1993} give the computed minimax type $(0,1)$ rational approximant and its maximum error as
$$\hat \xi_{\cal D}(x)\approx \frac{0.2993-{\tt i}\cdot 0.068}{x-1.1194  e^{{\tt i}\cdot 1.1490}},~~\|\hat \xi_{\cal D}(x)\|_{\infty,{\cal D}}\approx1.9274,$$
respectively, 
with extreme points at $e^{{\tt i}\varphi_j}\in \Gamma$ for $\varphi_1\approx 1.1335$, $\varphi_2\approx 3.3449$, and $\varphi_3\approx 6.0125$. Notably, Ruttan's sufficient global optimality \eqref{eq:RuttanHc} is not satisfied here \cite{this:1993}.
For the restriction to the boundary $\Gamma$, \cite[Example 5]{this:1993} establishes the uniqueness of the minimax type $(0,1)$ approximant:
$\hat \xi_{\Gamma}(x)=\frac{2/(\sqrt{33}+1)}{x}$
achieving $\|\hat \xi_{\Gamma}(x)\|_{\infty,{\Gamma}}=\sqrt{18/(69-11\sqrt{33})}\approx 1.7602$, with four extreme points on $\Gamma$ where Ruttan's sufficient global optimality \eqref{eq:RuttanHc} does hold \cite{this:1993}.

Based on this observation, we conclude that for any discrete set $\mathcal{X}\subset\Gamma$, the corresponding type $(0,1)$ minimax approximant $\hat{\xi}_{\mathcal{X}}$ (when it exists) cannot equal or converge to $\hat{\xi}_{\mathcal{D}}$. Specifically,
$$\eta_{\mathcal{X}} = \inf_{\xi\in\mathscr{R}_{(0,1)}} \|f-\xi\|_{\infty,\mathcal{X}}  
\leq \inf_{\xi\in\mathscr{R}_{(0,1)}} \|f-\xi\|_{\infty,\Gamma} \approx 1.7602
< \inf_{\xi\in\mathscr{R}_{(0,1)}} \|f-\xi\|_{\infty,\Omega} \approx 1.9274.$$

This example motivates a sufficient condition ensuring that the rational minimax approximant $\hat \xi_{\Omega}$ of $f\in {\bf C}_{A}(\Omega)\setminus \scrR_{(n_1,n_2)}$ on $\Omega$ can be obtained by solving a discrete approximation problem with proper sampled nodes ${\cal X} \subset \Gamma$. 
\begin{theorem}\label{thm:sufXtoD}
Let $\Omega \subset \mathbb{C}$ be a Jordan domain bounded by a   Jordan curve $\Gamma = \partial\Omega$, and consider $f\in {\bf C}_{A}(\Omega)\setminus \scrR_{(n_1,n_2)}$. Assume that the minimax approximant $\hat \xi_{\Omega}\in \scrR_{(n_1,n_2)}$  of \eqref{eq:bestf0} satisfies Ruttan's   sufficient global optimality \eqref{eq:RuttanHc}. Then we have 
\begin{itemize}
\item[(i)] $\hat \xi_{\Omega}\in \scrR_{(n_1,n_2)}$ is also a  type $(n_1,n_2)$ minimax approximant of $f$ on the boundary $\Gamma$;
\item[(ii)] if ${\cal X}=\{x_j\}_{j=1}^m\subset \Gamma$ contains a set of extreme points $\{z_j\}_{j=1}^r\subseteq {\cal E}(\hat \xi_{\Omega})\subset \Gamma$ associated with $\hat \xi_{\Omega}$ so that Ruttan's sufficient global optimality \eqref{eq:RuttanHc} holds, then $\hat \xi_{\Omega}\in \scrR_{(n_1,n_2)}$ is also a   type $(n_1,n_2)$ minimax approximant of $f$ on ${\cal X}$.
\end{itemize}
\end{theorem}
\begin{proof}
By the maximum modulus principle, the supremum norm $\|f - \hat \xi_{\Omega}\|_{\infty,\Omega}$ is attained on the boundary $\Gamma$. Thus, the result follows directly from Theorem \ref{thm:RuttanSufOpt}, taking ${\cal B} = \Gamma$ for part (i) and ${\cal B} = {\cal X}$ for part (ii).
\end{proof}
\section{Concluding remarks}\label{sec:conclude}

This paper presents a theoretical analysis of optimality conditions for rational minimax approximations under varying domains ${\cal B}$. We primarily examine Ruttan's optimality conditions and extend their applicability from discrete rational minimax approximations to continuum counterparts. Our key results include
\begin{itemize}
\item  Optimality conditions: We systematically reviewed and extended Ruttan's primal and dual formulations of second-order optimality conditions, validating their applicability to finite discrete sets (Section \ref{subsec:Ruttan2nd}).
For global minimax solutions, we proved in Theorem \ref{thm:Ruttansufnec} that Ruttan's sufficient condition becomes necessary when $|{\cal E}(\hat{\xi})| = n_1 + n_2 + 2 - \upsilon(\hat{p}, \hat{q})$, generalizing prior results (Section \ref{subsec:ruttanglobalNec}).

\item
Connection to numerical methods: We demonstrated that the recently proposed dual-based framework in the {\tt d-Lawson} iteration \cite{zhyy:2025,zhha:2025} has closed relations with both the Kolmogorov criteria and Ruttan's conditions, providing a theoretical foundation for  the {\tt d-Lawson} iteration.  Specifically, under strong duality \eqref{eq:strongdualityRuttan}, both Ruttan's sufficient condition \eqref{eq:RuttanHc} and the   Kolmogorov dual criteria \eqref{eq:kolvRD} hold at the global solution of the dual problem \eqref{eq:dual}.

\item
Continuum approximations: For continuum domains, we established in Theorem \ref{thm:sufXtoD} that when minimax solutions on the full domain satisfy Ruttan's condition, they can be recovered through discrete approximations  with  ${\cal X}$ sampled at appropriate boundary points. This fundamental connection enables computing continuum minimax approximants via practical discrete methods.

\end{itemize}
Our results offer new theoretical perspectives for the {\tt d-Lawson} iteration \cite{zhyy:2025,zhha:2025}. We hope these developments may lead to natural extensions:  by reformulating  the dual function \eqref{eq:rat-d-compt} in terms of integrals \cite[Section 2.3]{shap:1971}, the max-min duality framework \eqref{eq:dual} could potentially be adapted for  the rational minimax problems on a continuum, while Kolmogorov's and Ruttan's optimality conditions might also find meaningful generalizations to matrix-valued cases \cite{zhzz:2025}.
  
\def\noopsort#1{}\def\l{\char32l}\def\v#1{{\accent20 #1}}
  \let\^^_=\v\def\hbk{hardback}\def\pbk{paperback}
\providecommand{\href}[2]{#2}
\providecommand{\arxiv}[1]{\href{http://arxiv.org/abs/#1}{arXiv:#1}}
\providecommand{\url}[1]{\texttt{#1}}
\providecommand{\urlprefix}{URL }

\end{document}